\documentclass[11pt]{amsart}
   \usepackage{amsmath,amsthm}
   \usepackage{amsfonts}   % if you want the fonts
   \usepackage{amssymb}    % if you want extra symbols
   \usepackage{url}

\advance \topmargin by -\headheight
\advance \topmargin by -\headsep
      
\evensidemargin \oddsidemargin
\usepackage[all]{xy}

\newtheorem{thm}{Theorem}

\newtheorem{lem}[thm]{Lemma}
\newtheorem{cor}[thm]{Corollary}

\newtheorem{qu}[thm]{Question}

\title{The homology of surface diffeomorphism groups and a question of Morita}
\author{Jonathan Bowden}
\address{Mathematisches Institut, Ludwig-Maximilians-Universit\"at, Theresienstr. 39, 80333 M\"unchen, Germany}
\email{jonathan.bowden@mathematik.uni-muenchen.de}
\date{\today}

\begin{document}
\maketitle
\begin{abstract}
We answer a question posed by Morita concerning the non-triviality of certain secondary characteristic classes for surface bundles. In doing so we are naturally led to show that a form of Harer stability holds for surface diffeomorphism groups in homology of small degree.
\end{abstract}
\section{Introduction}
In this article we are concerned with the problem of computing the homology of surface diffeomorphism groups considered as discrete groups. The starting point for our discussion of the homology of surface diffeomorphism groups is the following problem posed by Morita:
\begin{qu}[\cite{Mor1}, Problem 12.1]\label{Morita}
Is the map 
\[H_3(Diff^{\delta}_+(\Sigma_g)) \to \mathbb{R}^2\]
induced by $u_1c_2$ and $u_1c_1^2$ a surjection for every $g$?
\end{qu}
\noindent An affirmative answer to this question yields a natural generalisation of a result of Rasmussen, who showed that the above map is surjective for $S^2$ (see \cite{Ras}). We will in fact show that the general case is a consequence of Rasmussen's result as an application the following theorem. For this we let $Diff_c(D^2)$ be the group of compactly supported diffeomorphisms whose support lies in the interior of the closed disc $D^2$.
\begin{thm}\label{hom_surj}
Let $D^2 \subset S^2$ be an inclusion of any closed disc. Then the induced map
\[H_k(Diff^{\delta}_c(D^2), \mathbb{Q}) \to H_k(Diff^{\delta}_+(S^2), \mathbb{Q})\]
is an isomorphism for $0 \leq k \leq 2$ and a surjection for $k=3$.
\end{thm}
For then the map 
\[H_3(Diff^{\delta}_c(D^2)) \to \mathbb{R}^2\]
induced by the classes $u_1c_2$ and $u_1c_1^2$ is surjective and factors through $H_3(Diff^{\delta}_+(\Sigma_g))$ for any $g$, which immediately yields a positive answer to Morita's question.

The proof of Theorem \ref{hom_surj} is based on certain calculations used to prove the perfectness of the identity component of the group of compactly supported diffeomorphism. These techniques have become a sort of folklore and as such we must fall back upon the PhD thesis of Stefan Haller (\cite{Hal}) for several technical facts that form the basis of our arguments (see also \cite{Mat}, \cite{Tsu}).

%Once these results have been assembled the proof of Theorem \ref{hom_surj} is fairly immediate.

Morita has also posed the problem of whether Harer stability holds for surface diffeomorphism groups (\cite{Mor1}, Problem 12.2). As a further application of the methods we discuss, we will show that a version of Harer stability does in fact hold in low degrees. In order to state this we let $\Sigma^1_g$ denote a compact genus $g$ surface with one boundary component and $\Gamma_g^1$ the mapping class group of diffeomorphisms with support in the interior of $\Sigma^1_g$ modulo isotopy. We further consider the natural inclusion
\[\Sigma^1_g \hookrightarrow \Sigma^1_{g+1}.\]
Harer stability then says that the inclusion above induces an isomorphism
\[H_k(\Gamma_g^1, \mathbb{Q}) \to H_k(\Gamma_{g+1}^1, \mathbb{Q}) \text{, if } g \geq \frac{3k}{2} + 2.\]
We will show that a similar property holds for the groups $Diff_c(\Sigma_g^1)$ of diffeomorphisms with support in the interior of $\Sigma^1_g$ up to homology of degree $3$ (cf.\ Theorem \ref{Harer_stability}). We further prove that the image of $H_4(Diff^{\delta}_c(\Sigma_g^1), \mathbb{Q})$ in $H_4(\Gamma_g^1, \mathbb{Q})$ is independent of the genus if $g \geq 8$. This result then has implications for the problem of determining the non-triviality of the image of the second Mumford-Miller-Morita (MMM) class $e_2$ in $H^4(Diff^{\delta}_c(\Sigma_g^1), \mathbb{Q})$, which was posed in \cite{KM1}. In particular, it implies that the non-vanishing of the image of $e_2$ in $H^4(Diff^{\delta}_c(\Sigma_g^1), \mathbb{Q})$ is independent of $g$ as soon as $g$ is at least $8$.

%Before giving proofs of the main results we gather the necessary results on the homology of $B \overline{G}$ in Section \ref{BG}. 

\section{Results on $B \overline{G}$}\label{BG}
In this section we shall recall some general results concerning the homology of the space $B \overline{G}$. For a general topological group $G$ we let $G^{\delta}$ be the same group considered with the discrete topology. The homotopy fibre of the map
\[G^{\delta} \to G\]
is also a topological group in a natural way and is denoted $\overline{G}$. By considering classifying spaces we obtain in this way a fibration 
\[B \overline{G} \to B G^{\delta} \to B G.\]
For our purposes $G = Diff_{c,0}(M)$ will always be the identity component of the compactly supported diffeomorphism group of a smooth manifold $M$ considered as a topological group with the $C^{\infty}$-topology.

The homology of $B \overline{G}$ is computed as the homology of the complex generated by smooth singular simplices $\sigma: \Delta_n \to G$ that send a base vertex $e_0$ to the identity (\cite{Hal}, Lemma 1.4.3). Recall that a simplex is smooth if the map 
\[(p,x) \mapsto \sigma(p)(x)\]
defines a smooth map $\Delta_n \times M \to M$. The associated complex will be simply denoted by $C_*(B \overline{G})$. If $\mathcal{U} = \{U_i\}$ is an open covering of $M$, we let $H^{\mathcal{U}}_*(B \overline{G})$ be the homology of the subcomplex $C^{\mathcal{U}}_*(B \overline{G})$ consisting of simplices each of whose supports lie in some $U_i$. If the covering consists of a single set $U$ we will write $H_*(B \overline{G_U})$ following \cite{Hal}. This notation reflects the fact that $H^{U}_*(B \overline{G})$ is just the homology of $B \overline{G_U}$, where $G_U = Diff_{c,0}(U)$. The most important result we shall need is the so-called fragmentation lemma.
\begin{lem}[\cite{Hal}, Th.\ 2.2.10]
Let $\mathcal{U}$ be any covering of a manifold $M$ and let $\mathcal{U}^{(m)}$ denote the cover consisting of all $m$-fold unions of sets in $\mathcal{U}$. Then the natural map
\[H^{\mathcal{U}^{(m)}}_k(B \overline{G}) \to H_k(B \overline{G})\]
is an isomorphism for all $0 \leq k \leq m$.
\end{lem}
Furthermore, there is a spectral sequence for computing $H^{\mathcal{U}}_*(B \overline{G})$ involving certain \v{C}ech homology groups. For this we choose an ordering on the index set of the covering $\mathcal{U} = \{U_i\}$ and we let 
\[U_{i_1,...,i_k} = U_{i_1} \cap ... \cap U_{i_k}.\]
\begin{lem}[\cite{Hal}, Cor.\ 2.3.2]\label{SpecSeq_1}
There is a spectral sequence converging to $H^{\mathcal{U}}_*(B \overline{G})$ whose $E^1$-page has terms
\[E_{p,q}^1 = \bigoplus_{i_1 < .. < i_p} H_q(B \overline{G_{U_{i_1,...,i_p}}})\]
and the differential is given by the \v{C}ech boundary map.
\end{lem}
The second page of this spectral sequence can be computed in the range of $0 \leq q \leq 2k - 1$, where $k$ is the least value for which $H_k(B \overline{Diff_{c}(\mathbb{R}^n)})$ is non-trivial and $n$ is the dimension of $M$. This fact is based on the following lemma, which is a slight extension of (\cite{Hal}, Lemma 1.4.9), and follows immediately from the K\"{u}nneth formula.
\begin{lem}\label{disjoint}
Let $M = M_1 \sqcup M_2$ be a disjoint union of $n$-manifolds and set $G_i = Diff_{c,0}(M_i)$. Assume that $H_n(B\overline{G_i})$ is trivial for all $n < k$. Then $H_0(B\overline{G}) = \mathbb{Z}$ and for all $1 \leq n \leq 2k - 1$
\[ H_n(B\overline{G}) \cong H_n(B\overline{G_1}) \oplus H_n(B\overline{G_2}).\]
\end{lem}
As a consequence of Lemma \ref{disjoint} we obtain a description of the $E^2$-terms of the spectral sequence in Lemma \ref{SpecSeq_1} that is associated to an \emph{admissible} covering. A covering is admissible if the finite intersection of any of its members is diffeomorphic to a disjoint union of balls. For any $U$ that is diffeomorphic to a ball, we let
\[A_q^n = H_q(B \overline{G_U}) = H_q(B \overline{Diff_{c}(\mathbb{R}^n)}).\]
With this notation we have the following stronger version of Lemma \ref{SpecSeq_1}, which is a mild strengthening of Theorem 2.3.4 in \cite{Hal} and whose proof is identical.
\begin{lem}\label{SpecSeq_2}
Let $\mathcal{U}$ be an admissible covering of $M$ and let $\mathcal{U}^{(m)}$ denote the cover consisting of all $m$-fold unions of sets in $\mathcal{U}$. Then there is a spectral sequence converging to $H^{\mathcal{U}^{(m)}}_*(B \overline{G})$ whose $E^2$-page has terms
\[E_{p,q}^2 = H_p(M, A_q^n)\]
for $1 \leq q \leq 2k -1 \leq m$ with $k$ as in Lemma \ref{disjoint}. Moreover, $E_{0,0}^2 = \mathbb{Z}$ and $E_{p,0}^2 = 0$ for $p > 0$.
\end{lem}
%\begin{rem}
The entire diffeomorphism group $Diff_c(M)$ acts on its identity component via conjugation and this action induces a module structure on $H_*(B\overline{G})$. Since the identity component of $Diff_c(M)$ acts trivially, this then descends to a module structure for the mapping class group $\Gamma_M$ of $M$. This module structure is then compatible with the spectral sequence given in Lemma \ref{SpecSeq_2}, since any $\phi \in Diff_c(M)$ induces a natural map between the spectral sequences associated to $\mathcal{U}^{(m)}$ and $\phi(\mathcal{U}^{(m)})$ respectively. In particular, since $\Gamma_M$ acts trivially on $A_q^n$, the action on terms of the form $H_p(M, A_q^n)$ is given via the ordinary action of the mapping class group on $H_p(M)$.
%\end{rem}

%\begin{proof}
%By Lemma \ref{SpecSeq_1} there is a spectral sequence converging to $H^{\mathcal{U}^{(m)}}_*(B \overline{G})$, whose $E^1$-term is
%\[E_{p,q}^1 = \bigoplus_{i_1 < .. < i_p} H_q(B \overline{G_{U_{i_1,...,i_p}}}),\]
%where $U_{i_1,...,i_p}$ is an intersection of members of $\mathcal{U}^{(m)}$. Since $\mathcal{U}$ is an admissible covering, this implies that $U_{i_1,...,i_p}$ is a disjoint union of balls
%\[ U_{i_1,...,i_p} = \bigsqcup^N_{j=1} (B_{i_1,...,i_p})_j\]
%so that $\mathcal{U}^{(m)}$ is admissible. For $0 \leq p \leq 2k -1$ Lemma \ref{disjoint} implies
%\[H_q(B \overline{G_{U_{i_1,...,i_p}}}) \cong H_q(B_{i_1,...,i_p})_1 \oplus...\oplus  H_q(B_{i_1,...,i_p})_N \cong H_0(U_{i_1,...,i_p}, A^n_q),\]
%and this isomorphism is canonical since $Diff_{c}(\mathbb{R}^n)$ is connected. It then follows by Lemma 2.4.2 in \cite{Hal} that
%\[E_{p,q}^2 = H_p(M,A_q^n). \qedhere\]
%\end{proof}

\section{Proofs of main results}
We now come to the proof of Theorem \ref{hom_surj}.
\begin{proof}[Proof of Theorem \ref{hom_surj}]
By the classical result of Thurston (see \cite{Th2})
\[A_1^n = H_1(B \overline{Diff_{c,0}(M)}) = H_1(B \overline{Diff_{c}(\mathbb{R}^n)}) =0.\]
Specialising to the case of $M = S^2$ and considering the spectral sequence of Lemma \ref{SpecSeq_2} we see that the $E^2$-page in degree 3 homology is
\[E_{0,3}^2 \oplus E_{1,2}^2 = H_0(S^2, A_3^2) \oplus H_1(S^2, A_2^2) = A_3^2.\]
Thus the $E^{\infty}$-term in degree 3 is a quotient of this group. If we consider the inclusion of any disc $U = D^2$ into $S^2$, then a comparison of the associated spectral sequences shows that the map
\[H_k(B\overline{Diff_{c}(D^2)}) \to H_k(B\overline{Diff_+(S^2)})\]
is an isomorphism for $0 \leq k \leq 2$ and a surjection for $k=3$. Finally, the topological group $G = Diff_+(S^2)$ is homotopy equivalent to $SO(3)$ by Smale's Theorem (cf.\ \cite{Iva}). Thus the inclusion of the fibre in the fibration
\[ B \overline{G} \to B G^{\delta} \to B G \simeq BSO(3)\]
induces an isomorphism in rational homology for $k \leq 3$, since $H_k(BSO(3), \mathbb{Q})$ is trivial for $k \leq 3$ and this concludes the proof.
\end{proof}
We then obtain an answer to Question \ref{Morita} as a corollary.
\begin{cor}\label{Morita_question}
For all $g \geq 0$ the classes $u_1c_2$ and $u_1c_1^2$ induce a surjective map
\[H_3(Diff^{\delta}_+(\Sigma_g)) \to \mathbb{R}^2.\]
\end{cor}
\begin{proof}
We consider the following commutative diagram, where the maps to $\mathbb{R}^2$ are those induced by the classes $u_1c_2$ and $u_1c_1^2$
\[\xymatrix{ H_3(Diff^{\delta}_+(S^2))  \ar[r]& \mathbb{R}^2. \\
 H_3(Diff^{\delta}_c(D^2)) \ar[u] \ar[ur] &}\]
Since the map on left is surjective modulo torsion by Theorem \ref{hom_surj} and the top-most map is surjective by \cite{Ras}, we conclude that the map on $ H_3(Diff^{\delta}_c(D^2))$ is also surjective. Then by considering the analogous diagram for $Diff^{\delta}_+(\Sigma_g)$, the result follows.
%We let $\sigma_{\lambda_1, \lambda_2}$ be a class in $H_3(Diff^{\delta}_+(S^2))$ be such that 
%\[\psi_i(\sigma_{\lambda_1, \lambda_2}) = \lambda_i\]
%as given by \cite{Ras}. Then by Theorem \ref{hom_surj}, $\sigma_{\lambda_1, \lambda_2}$ comes from a class $\bar{\sigma}_{\lambda_1, \lambda_2} \in H_3(Diff^{\delta}_{c}(D^2))$. Moreover, we may assume that $\bar{\sigma}_{\lambda_1, \lambda_2}$ is represented by a manifold $M$, in which case we may think of $\bar{\sigma}_{\lambda_1, \lambda_2}$ as given by a horizontal foliation $\mathcal{F}_{\lambda_1, \lambda_2}$ on $M \times D^2$. We then compute:
%\[\psi_1(\bar{\sigma}_{\lambda_1, \lambda_2}) = \pi_{!}u_1c_1^2 (\bar{\sigma}_{\lambda_1, \lambda_2}) = \int_{M \times D^2} u_1c_1^2 = \lambda_1\]
%and
%\[\psi_2(\bar{\sigma}_{\lambda_1, \lambda_2}) = \pi_{!}u_1c_2 (\bar{\sigma}_{\lambda_1, \lambda_2}) = \int_{M \times D^2} u_1c_2 = \lambda_2.\]
%Then for the image $\bar{\sigma}^{\Sigma_g}_{\lambda_1, \lambda_2}$ of $\bar{\sigma}_{\lambda_1, \lambda_2}$ in $H_3(Diff^{\delta}_+(\Sigma_g))$ the horizontal foliations are such that $u_1c_1^2(\mathcal{F}_{\lambda_1, \lambda_2})$ and $u_1c_2(\mathcal{F}_{\lambda_1, \lambda_2})$ have support in $M \times D^2$. Thus, we compute:
%\[\psi_1(\bar{\sigma}^{\Sigma_g}_{\lambda_1, \lambda_2}) = \int_{M \times \Sigma_g} u_1c_1^2 = \int_{M \times D^2} u_1c_1^2  = \lambda_1\]
%and
%\[\psi_2(\bar{\sigma}^{\Sigma_g}_{\lambda_1, \lambda_2}) = \int_{M \times \Sigma_g} u_1c_2  = \int_{M \times D^2} u_1c_2 =\lambda_2,\]
%proving the Corollary. 
\end{proof}

We next turn to the proof of low-dimensional Harer stability for the groups $Diff_{c}(\Sigma^1_g) $.
\begin{thm}\label{Harer_stability}
The natural map $\Sigma^1_g \hookrightarrow \Sigma^1_{g+1}$ induces an isomorphism
\[H_k(Diff^{\delta}_c(\Sigma_g^1), \mathbb{Q}) \to H_k(Diff^{\delta}_c(\Sigma_{g+1}^1), \mathbb{Q})\]
for $k \leq 3$ and  $g \geq 8$. Furthermore, the rank of the image of $H_4(Diff^{\delta}_+(\Sigma_g^1), \mathbb{Q})$ in $H_4(\Gamma_g^1, \mathbb{Q})$ is independent of $g$ for $8 \leq g$.
\end{thm}
\begin{proof}
We first note that by \cite{ES} the topological group $Diff_{c,0}(\Sigma^1_g)$ is contractible for any $g$. Thus
\[B\overline{Diff_{c,0}(\Sigma^1_g)} \simeq B Diff^{\delta}_{c,0}(\Sigma^1_g),\]
which in turn yields a natural isomorphism 
\[H_* (B\overline{Diff_{c,0}(\Sigma^1_g)} ) \cong H_*(Diff^{\delta}_{c,0}(\Sigma^1_g)).\]
We consider the Hochschild-Serre spectral sequence associated to the group extension
\[1 \to Diff_{c,0}(\Sigma^1_g) \to Diff_{c}(\Sigma^1_g) \to \Gamma^1_g \to 1,\]
whose $E^2$-page is
\[E^2_{p,q} = H_p(\Gamma^1_g, H_q(Diff^{\delta}_{c,0}(\Sigma^1_g), \mathbb{Q})).\]
When referring explicitly to the $k$-th page of the spectral sequence associated to the extension $Diff_{c}(\Sigma^1_g)$ we shall write $E^k_{p,q}(\Sigma^1_g)$.

Since $H_* (B\overline{Diff_{c,0}(\Sigma^1_g)} ) \cong H_*(Diff^{\delta}_{c,0}(\Sigma^1_g))$, Theorem \ref{hom_surj} and the naturality of the Hochschild-Serre spectral sequence imply that the following is an isomorphism for $q \leq 2$:
\[E^2_{0,q}(\Sigma_g^1) \to E^2_{0,q}(\Sigma_{g+1}^1).\]
Moreover, Harer stability for $\Gamma^1_g$ yields isomorphisms on $E^2_{p,0}$ terms for any $p \leq 4$ and the groups $E^2_{p,1}$ are all trivial, since $Diff_{c,0}(\Sigma^1_g)$ is perfect. We conclude that the maps
\[E^2_{p,q}(\Sigma_g^1) \to E^2_{p,q}(\Sigma_{g+1}^1)\]
are isomorphisms for all bi-degrees $(p,q)$ such that $p + q \leq 2$ or such that $p \leq 4$ and $q = 0$.
Similarly we have isomorphisms on $E^2_{2,2}$-terms for $ g \geq 5$. For Lemma \ref{SpecSeq_2} implies 
\[E^2_{2,2}(\Sigma_g^1) \cong H_2(\Gamma^1_g, A^2_2) \cong E^2_{2,2}(\Sigma_{g+1}^1),\]
where the second isomorphism follows from  Harer stability since $A^2_2$ is a trivial $\Gamma^1_g$-module. We finally claim that
\[E^2_{0,3}(\Sigma_g^1) \to E^2_{0,3}(\Sigma_{g+1}^1)\]
is also an isomorphism and this implies the result in degree $3$ homology. To this end, we first note that by Lemma \ref{SpecSeq_2}
\[ H_3(Diff^{\delta}_{c,0}(\Sigma^1_g),\mathbb{Q}) \cong H_3(B\overline{Diff_{c,0}(\Sigma^1_g)},\mathbb{Q} ) \cong (H_0(\Sigma^1_g) \otimes A^2_3 \otimes \mathbb{Q}) \oplus (H_1(\Sigma^1_g) \otimes A^2_2\otimes \mathbb{Q}),\]
whence we conclude that
\[H_0(\Gamma_g^1, H_3(Diff^{\delta}_{c,0}(\Sigma^1_g),\mathbb{Q})) = H_3(Diff^{\delta}_{c,0}(\Sigma^1_g),\mathbb{Q})_{\Gamma_g^1}\cong A^2_3 \otimes \mathbb{Q}.\]
Moreover, since $\Gamma_g^1$ is perfect for $g \geq 2$
\[H_1(\Gamma_g^1, H_2(Diff^{\delta}_{c,0}(\Sigma^1_g),\mathbb{Q})) = H_1(\Gamma_g^1,A_2^2\otimes \mathbb{Q}) = 0\]
proving the claim and with it the first part of the theorem. 

In fact, since $H_3(\Gamma_g^1, \mathbb{Q})$ is trivial for $g \geq 6$ (see \cite{Iva}), we deduce that the only non-trivial $E^2$-term in degree $3$ is $E^2_{0,3}$. The isomorphism on $E^2_{2,2}$ then yields an isomorphism on $E^3_{0,3}$ and then the fact that $E^2_{3,1} = E^3_{3,1} = 0$ implies the same for $E^4_{0,3}$. We then consider the following commutative diagram:
\[\xymatrix{E^4_{0,3}(\Sigma_g^1) \ar[r]^{\cong} & E^4_{0,3}(\Sigma_{g+1}^1)\\
H_4(\Gamma_g^1) = E^4_{4,0}(\Sigma_g^1) \ar[u]^{\partial^{4}_{4,0}} \ar[r] & \ar[u]^{\partial^{4}_{4,0}} H_4(\Gamma_{g + 1}^1) = E^4_{4,0}(\Sigma_{g+1}^1),}\]
where the bottom arrow is an isomorphism for $g \geq 8$ by Harer Stability and $\partial^{4}_{4,0}$ denotes the bi-graded differential on the $E^4$-page of the spectral sequence. The image of $H_4(Diff^{\delta}_+(\Sigma_g^1), \mathbb{Q})$ in $H_4(\Gamma_g^1, \mathbb{Q})$ is $Ker(\partial^{4}_{4,0})$, whose rank is then independent of $g \geq 8$.
\end{proof}
This result has implications for a problem posed by Kotschick and Morita. In particular, they asked whether the image of the second MMM-class in $H^*(\Gamma^1_g)$ is non-trivial (\cite{KM1}, Problem 4). Theorem \ref{Harer_stability} says that the answer to this question is independent of $g$. In fact a closer examination of the proof of Theorem \ref{Harer_stability} shows that the same conclusion holds for closed surfaces and for surfaces with arbitrarily many boundary components.

\section*{Acknowledgements}
The author would like to thank first and foremost Prof.\ D.\ Kotschick for his continued encouragement and support. The financial support of the Deutsche Forschungsgemeinschaft is also gratefully acknowledged.

\end{document}